\documentclass[11pt]{article}

\usepackage[top=1in,bottom=1in,left=1in,right=1in]{geometry}
\usepackage{amsmath,amsthm,amssymb}
\usepackage{mathtools}
\usepackage{enumerate}
\usepackage{caption}
\usepackage{xcolor}
\usepackage{tikz}
\usetikzlibrary{positioning,arrows.meta,shapes.geometric}
\usepackage[colorlinks=true,linkcolor=blue,citecolor=blue,urlcolor=blue]{hyperref}

\newtheorem{theorem}{Theorem}[section]
\newtheorem{corollary}[theorem]{Corollary}
\newtheorem{proposition}[theorem]{Proposition}
\newtheorem{lemma}[theorem]{Lemma}
\newtheorem{remark}[theorem]{Remark}

\theoremstyle{definition}
\newtheorem{definition}[theorem]{Definition}

\newcommand{\GG}{\mathcal{G}}
\newcommand{\R}{\mathbb{R}}
\DeclareMathOperator{\sgn}{sgn}

\numberwithin{equation}{section}

\begin{document}

\title{Instantaneous arithmetic computation via ratio-encoding in chemical reaction networks}

\author{David F. Anderson\footnote{University of Wisconsin-Madison (anderson$@$math.wisc.edu)}~ and Badal Joshi\footnote{Department of Mathematics, California State University San Marcos (bjoshi$@$csusm.edu). Corresponding author.}}

\date{}

\maketitle

\begin{abstract}

We develop a general framework for \emph{instantaneous} arithmetic computation
using chemical reaction networks. Numerical values are encoded by ratios of
species concentrations:  an extended nonnegative value $a\in[0,+\infty]$ is represented by a
computational pair $(A_0,A_1)$ through the ratio $a=a_1/a_0$. For this ratio
encoding, our main theoretical result is a feedforward compositionality theorem:
if a collection of modules satisfies two easily checked structural conditions,
then any admissible finite feedforward composition computes the corresponding arithmetic
expression instantaneously, meaning that every intermediate and output ratio is
correct at every positive time.

We make this framework concrete by constructing reaction network modules for
four elementary operations---identification, inversion, multiplication, and
addition---and verifying that each satisfies these structural conditions.  As applications, we show that truncated
power series and matrix products can be computed instantaneously using reaction
networks whose size reflects the number of arithmetic operations in the
underlying computation. We also extend the nonnegative ratio encoding to signed
quantities by representing a real number as the difference of two nonnegative
ratio-encoded values; this allows addition and multiplication over real-valued
inputs. We discuss the practical implications of instantaneous computation,
including its relationship to other resource constraints such as the number of
chemical species required and the cost of reading out the final answer, which
requires division and, in the signed case, rectified subtraction.

\end{abstract}
{\bf Keywords:} molecular computation, computing with reaction networks, chemical computation, fractional representation, analog computation, computational speed.

\section{Introduction}

Recent theoretical and technological advances have made it possible to implement computation using designed chemical reaction networks \cite{anderson2021reaction, cappelletti2020stochastic, chen2023rate, chen2014deterministic, qian2011neural, QSW2011, qian2011simple, soloveichik2010dna}. A driving motivation for the present work is to keep mathematical theory in step with emerging hardware platforms --- such as DNA strand displacement technology \cite{soloveichik2010dna, qian2011scaling} --- so that algorithmic designs are available when physical implementations become feasible. In such computational designs, fundamental arithmetic operations such as addition and multiplication must be executed numerous times in any composite computation, and their speed is therefore of critical importance.

The literature on CRN computation is broad; we recall only the strands most relevant to the present work.
The computational power of CRNs has been studied from a number of angles. Among results in the discrete stochastic setting, Chen, Doty, and Soloveichik \cite{chen2014deterministic} showed that the class of functions deterministically computable by stochastic CRNs is precisely the semilinear functions, while Cappelletti, Ortiz-Mu\~noz, Anderson, and Winfree \cite{cappelletti2020stochastic} established that stochastic CRNs can approximate any discrete probability distribution with arbitrary accuracy at stationarity. In the continuous deterministic setting, Fages, Le~Guludec, Bournez, and Pouly \cite{fages2017strong} proved strong Turing-completeness of continuous CRNs; Chen, Doty, Reeves, and Soloveichik \cite{chen2023rate} showed that rate-independent computation by continuous CRNs corresponds exactly to the piecewise rational linear functions; and Anderson, Joshi, and Deshpande \cite{anderson2021reaction} constructed mass-action CRNs that implement arbitrary neural networks. 
Doty, Latifi, and Soloveichik \cite{doty2025analog} used a ratio representation similar to the one used in this paper for a different application --- to show that any mass-action ODE system can be implemented by a transcriptional network in which the only negative terms are linear degradation at a rate common to all species.
On the hardware side, the work of Soloveichik, Seelig, and Winfree \cite{soloveichik2010dna} demonstrated that DNA strand displacement can physically realize any CRN \cite{qian2011scaling, qian2011neural}, providing a concrete experimental platform for these theoretical constructions. 

In \cite{anderson2025arithmetic}, we introduced the concept of \emph{input-independent computational speed} for mass-action CRNs and showed that a suite of fundamental arithmetic operations can be computed within a time that is independent of the input values and scales only with the desired accuracy. Furthermore, these operations can be composed to produce computations of arbitrary complexity without any increase in computational time. The research in \cite{anderson2026computing} extends this program to the transcendental functions $e^x$ and $\ln x$, constructing reaction network modules that compute them directly --- without relying on power series --- and proving that any composite computation combining these with the arithmetic modules of \cite{anderson2025arithmetic} also runs at input-independent speed. In the present paper we pursue a different and in some ways more striking goal: \emph{instantaneous} computation, meaning that the encoded output is correct for \emph{all} $t > 0$, not merely asymptotically as $t \to \infty$.

To do this, we take inspiration from a series of papers by Salehi, Liu, Riedel,
Parhi, and collaborators \cite{salehi2017chemical, salehi2018computing, solanki2023computing}, which
developed a fractional-encoding approach to molecular computation motivated by
stochastic logic. 
Since our construction is best understood as a modification of
that viewpoint, we briefly recall their framework and then explain why our
present goal requires a different encoding.
 In stochastic logic, a number in
$[0,1]$ is interpreted as the probability that a random bit is equal to one;
logic gates then implement arithmetic operations at the level of these
probabilities.  Salehi et al.\ translated this viewpoint into deterministic
chemical reaction networks by representing a value using a pair of species.
Specifically, a pair of chemical species $(A_0,A_1)$ with concentrations
$(a_0,a_1)\in\mathbb R^2_{\ge 0}$ encodes a real number $a\in[0,1]$ via the map
\begin{equation} \label{fractional_rep_hist}
(a_0, a_1) \mapsto \frac{a_1}{a_0 + a_1} = a \in [0,1].
\end{equation}
Under this encoding, the interval $[0,1]$ is closed under multiplication $(a,b)\mapsto ab$, its complement $(a,b)\mapsto 1-ab$ and convex linear combination
$(a,b)\mapsto \sigma a+(1-\sigma)b$ for $\sigma\in[0,1]$.  
The authors of \cite{salehi2018computing} construct reaction network modules,
including multiplication and multiplexing/scaled-addition modules, that
implement these stochastic-logic operations, and they use these modules to build
CRNs for approximating more complicated functions. Crucially for the present
paper, these constructions are \textit{instantaneous} in the sense that the
relevant output fraction is correct for all $t>0$, not only asymptotically.
Thus, their work provides a systematic molecular implementation of
stochastic-logic computations using instantaneous fractional concentration
encodings.

For the stochastic-logic computations considered in
\cite{salehi2017chemical, salehi2018computing, solanki2023computing}, this representation is natural
and effective. For the purposes of the present paper, however, it leaves three
limitations that motivate a different ratio encoding.
\begin{enumerate}[(i)]
\item \emph{Restricted representation.} The representation \eqref{fractional_rep_hist} represents values in $[0,1]$;
Salehi et al.\ also discuss a related signed version representing values in
$[-1,1]$.

\item \emph{Non-standard arithmetic primitives.} The natural operations in
the stochastic-logic framework are complement, multiplication, and convex
linear combination, rather than the usual arithmetic operations of addition,
multiplication, and inversion.  In particular, $[0,1]$ is not closed under
ordinary addition, so standard arithmetic expressions and power series cannot
be represented directly without additional reformulation or scaling.

\item  \emph{Direct source loading.} 
Given an externally supplied $a\in[0,1]$, the natural normalized fractional pair $(a_0,a_1)=(1-a,a)$ requires forming the complement $1-a$ (and other choices likewise require a transformation of $a$). Under the ratio encoding used here, the same value is loaded directly as $(a_0,a_1)=(1,a)$, with no arithmetic preprocessing.
\end{enumerate}

We therefore modify the fractional-encoding idea by replacing the ratio $a_1/(a_0+a_1)$ with the ratio $a_1/a_0$: a pair $(A_0,A_1)$ with concentrations $(a_0,a_1)$ represents the value $a=a_1/a_0\in[0,+\infty]$. 
Under this encoding, the standard source convention $a_0(0)=1$, $a_1(0)=a$ requires no precomputation (Section~\ref{sec:fundamental}).
  We construct reaction network modules for four standard arithmetic operations --- identification, inversion, multiplication, and addition --- each of which computes instantaneously, and we describe their behavior at boundary values involving $+\infty$, including the explicitly stated cases in which a module is undefined (Section~\ref{sec:fundamental}).  We provide a rigorous mathematical framework showing that these modules can be composed arbitrarily to produce a computation of any complexity while preserving instantaneous computation throughout (Section~\ref{sec:composability}). The representation extends to all of $[-\infty, \infty]$ via a signed dual rail encoding (Section~\ref{sec:real}). As a consequence, truncated power series and matrix products can be computed instantaneously by reaction networks (Section~\ref{sec:power_series}). We also show that the structural conditions underlying our framework admit multiple implementations, offering flexibility for hardware constraints (Section~\ref{sec:alt_modules}). Finally, we discuss the practical implications of instantaneous computation, including its relationship to earlier work on input-independent speed and the role of output decoding in determining overall computation time (Section~\ref{sec:discussion}).   Throughout, the equations are written in nondimensional time and concentration variables, and all displayed dimensionless rate constants are chosen to be one; relaxing this normalization is a natural direction for future work.

One unavoidable feature of the ratio encoding approach is that the final output must be read off by an interpreter --- a step that lies outside the reaction network itself. For computations over $[0,+\infty]$, the output is a pair $(x_0, x_1)^T$ and decoding requires computing the ratio $x_1/x_0$, with the convention that $x_0 = 0$ is interpreted as $+\infty$. For computations over $[-\infty,+\infty]$, the output is a quadruple $(x_{p0}, x_{p1}, x_{n0}, x_{n1})^T$ and decoding requires two divisions followed by a subtraction, together with the appropriate handling of infinity cases. If the output interpretation is performed chemically, the operations needed --- division and rectified subtraction --- are available at input-independent speed via the design in \cite{anderson2025arithmetic}, making that paper a natural companion for the chemical decoding step. If the output interpretation is performed digitally, an efficient analog-digital interface is required to measure the relevant concentrations and perform the arithmetic on a digital platform.

\section{Fundamental design principles}
\label{sec:fundamental}

In this section we introduce the basic objects used throughout the paper.  We
first define, in Section \ref{sec:computational_pair}, computational pairs and the ratio encoding that represents elements
of $[0,+\infty]$.  In Section \ref{sec:computational_module_and_graph} we then describe computational modules and computational
graphs, which provide the language for composing elementary operations.  Next, in Section \ref{sec:source_value_encoding},
we specify how externally supplied source values are encoded as initial
conditions.  Finally, in Section \ref{sec:module_intro}, we present the four elementary reaction network modules
used in the rest of the paper: identification, inversion, multiplication, and
addition.  The verification that these modules have the claimed ratio-level
behavior is deferred to Section~\ref{sec:composability}.
Later, in Section~\ref{sec:real}, we extend the nonnegative ratio encoding introduced in this section to signed
quantities using a signed dual rail construction.

\subsection{Computational pair}
\label{sec:computational_pair}

A \emph{computational pair} is an ordered pair of chemical species
\[
Z := (Z_0,Z_1).
\]
We denote their concentrations at time $t$ by
\[
z(t) := (z_0(t),z_1(t))^T.
\]
We associate an element of $[0,+\infty]$ with this pair as follows: 
\begin{equation} \label{fractional_rep}
    (z_0, z_1) \in \R^2_{\ge 0} \mapsto z = 
    \frac{z_1}{z_0} \in [0, \infty], 
\end{equation}
assuming that at least one of $z_0$ or $z_1$ is positive. 
We associate $z = +\infty$ with any pair for which $z_0 = 0$ and $z_1 > 0$.
The individual concentrations $z_0$ and $z_1$ have no numerical significance in isolation; only their ratio $z_1/z_0$ carries meaning. Accordingly, we treat the pair $(Z_0,Z_1)$ as an inseparable computational unit.

We will say that a computational pair has the \emph{ratio invariance} property
on $(0,\infty)$ if the value associated with $(z_0(t),z_1(t))$ by
\eqref{fractional_rep} is well-defined and independent of $t$ for all $t>0$.
 In this case, the
pair represents a single numerical value throughout the positive-time evolution.

This property need not hold at $t=0$. In particular, we allow the initial
concentration vector of a computational pair to be
\[
(z_0(0),z_1(0))^T=(0,0)^T,
\]
in which case no numerical value is associated with the pair at time $0$.

\subsection{Computational module and computational graph}
\label{sec:computational_module_and_graph}

A {\em computational module} is a collection of two or more computational pairs with the following properties:
\begin{enumerate}[(i)]
    \item exactly one pair is designated as the output and the remaining (one or more) pairs designated as inputs, 
    \item a computational module is designed to compute a single elementary operation on the values encoded by the inputs and store the resulting value in the output pair.
\end{enumerate}
 
In this paper, we discuss modules that compute the operations of {\em identification}, {\em inversion}, {\em multiplication}, and {\em addition}.

A composite computation, made by composing elementary computational modules, is organized as a directed graph. 

\begin{definition}[Computational graph]
\label{def:computational_graph}
The \textit{computational graph} of a composite computation is a directed graph $\GG=(V,E)$ with the following properties:
\begin{itemize}
\item Each vertex $u\in V$ has a computational pair $(U_0,U_1)$ associated to it. 
\item The graph is \textit{acyclic} (no directed feedback loops).
\item A directed edge $(u\to v)\in E$ exists if and only if the computational
pair associated with $u$ is an input to the module associated with $v$.

\item A \textit{source vertex} is a vertex with in-degree zero; its computational
pair encodes an externally supplied source value.

\item Each non-source vertex $v$ is assigned a computational module $M_v$. The computational pair associated with $v$ is the output of this module, and the immediate predecessors of $v$ provide its inputs. Thus the in-degree of $v$ is the number of input pairs required by $M_v$.
\end{itemize}
\end{definition}

\begin{definition}[Admissibility]
\label{def:admissible}
A source-value assignment (or simply an input assignment) for a computational graph $\GG$ is \emph{admissible} if, proceeding through the graph in topological order, every non-source vertex receives input values for which its assigned module is defined. We also say that $\GG$ is \emph{admissible for that input assignment}. Thus, admissibility simply excludes undefined module evaluations.
\end{definition}

\begin{lemma}
     Suppose that an admissible input assignment is given at the source vertices of $\GG$. Then a unique value in $[0, + \infty]$ is assigned to every vertex in $\GG$, determined by applying the module $M_v$ at each non-source vertex $v$ to the values at its immediate predecessor vertices.
\end{lemma}

\begin{proof}
Since $\GG$ is acyclic, its vertices admit a topological ordering. Process the vertices in this order: source vertices carry their prescribed values. For each non-source vertex $v$, its value is uniquely determined by applying the operation of $M_v$ to the already-assigned values at its immediate predecessor vertices. By admissibility, each such module evaluation is defined. Since each non-source vertex has a unique producer, the assignment is well-defined and independent of the choice of topological ordering. The result follows by induction.
\end{proof}

\subsection{Source value encoding}
\label{sec:source_value_encoding}

A composite computation begins with a collection of externally specified  extended nonnegative values. These \emph{source values} are not produced by arithmetic
modules; rather, they are supplied through initial conditions on their
corresponding computational pairs. 

For a source value $a\in[0,+\infty]$, we encode $a$ by a computational pair
$(A_0,A_1)$ with concentration vector $(a_0(t),a_1(t))^T$. Since a ratio
representation is not unique, we fix the following convention.

\begin{definition}
Given a source value $a \in [0,+\infty]$, the \emph{standard encoding} is
defined by
\[
a_1(0)=a,\qquad a_0(0)=1 \qquad \text{if } a\in[0,\infty),
\]
and by
\[
a_1(0)=1,\qquad a_0(0)=0 \qquad \text{if } a=+\infty.
\]
\end{definition}

The usual extended-real conventions apply where the corresponding module is
defined: $a+\infty=+\infty$ for $a\in[0,+\infty]$, and
$a\cdot\infty=+\infty$ for $a\in(0,+\infty]$, while $0\cdot\infty$ is
undefined. The ratio representation has additional boundary
limitations: although $\infty+\infty$ is well-defined in the extended reals, the addition module leaves this case undefined (see Remark~\ref{rem:add_infty}). 
The undefined cases just described are exactly the module-specific boundary cases excluded by admissibility.

\subsection{Basic arithmetic modules}
\label{sec:module_intro}

We present four fundamental arithmetic modules which can be further combined to compute power series, matrix multiplications, and determinants.
In this section, we only give the basic modules; a complete justification of their operation and correctness of their output is provided in the next section.
Throughout, we assume mass action kinetics in nondimensional time and concentration variables, with every displayed dimensionless rate constant chosen to equal $1$.

\subsubsection{Identification module}

The value of $a \in [0, \infty]$ can be assigned from one computational pair to another  by the following reaction network: 

\begin{equation} \label{net:identification}
\begin{aligned}
    A_0 \to X_0, \quad \quad 
    A_1 \to X_1.
\end{aligned}
\end{equation}
The input is the computational pair $(A_0,A_1)$, representing $a = a_1/a_0$, and the output is the computational pair $(X_0,X_1)$, representing $x = x_1/x_0 = a$.

\subsubsection{Inversion (reciprocal) module}

The reciprocal of $a \in [0, \infty]$ is computed by the following reaction network:

\begin{equation} \label{net:inversion}
\begin{aligned}
    A_0 \to X_1, \quad \quad 
    A_1 \to X_0.
\end{aligned}
\end{equation}
The input is the computational pair $(A_0,A_1)$, representing $a = a_1/a_0$, and the output is the computational pair $(X_0,X_1)$, representing $x = x_1/x_0 = 1/a$.

\subsubsection{Multiplication module}

We define the {\em (ratio-encoded) multiplication map}, denoted $\otimes$: 
\begin{equation} \label{def:otimes}
\begin{pmatrix} a_0 \\ a_1 \end{pmatrix} \otimes \begin{pmatrix} b_0 \\ b_1 \end{pmatrix} = \begin{pmatrix} a_0 b_0 \\ a_1 b_1 \end{pmatrix}.
\end{equation}
The following reaction network module performs this multiplication $(a,b) \mapsto x = ab$, with the inputs encoded as $a \equiv (a_0,a_1)^T$ and $b \equiv (b_0,b_1)^T$ and the output as $x \equiv (x_0, x_1)^T$.  
\begin{equation} \label{net:multiplication}
\begin{aligned}
    A_0 + B_0 \to X_0, \quad 
    A_1 + B_1 \to X_1, \quad
    A_0 + B_1 \to 0, \quad 
    A_1 + B_0 \to 0.  
\end{aligned}
\end{equation}
The inputs are the computational pairs $(A_0,A_1)$ and $(B_0,B_1)$, representing $a = a_1/a_0$ and $b = b_1/b_0$ respectively, and the output is the computational pair $(X_0,X_1)$, representing $x = x_1/x_0 = a \cdot b$.

\subsubsection{Addition module}

We define the {\em (ratio-encoded) addition map}, denoted $\oplus$: 
\begin{equation} \label{def:oplus}
\begin{pmatrix} a_0 \\ a_1 \end{pmatrix} \oplus \begin{pmatrix} b_0 \\ b_1 \end{pmatrix} = \begin{pmatrix} a_0 b_0 \\ a_0 b_1 + a_1 b_0 \end{pmatrix}.
\end{equation}
The following module performs this addition $(a,b) \mapsto x = a + b$.  
\begin{equation} \label{net:addition}
\begin{aligned}
    A_0 + B_0 \to X_0, \quad
    A_0 + B_1 \to X_1, \quad
    A_1 + B_0 \to X_1, \quad
    A_1 + B_1 \to 0.
\end{aligned}
\end{equation}
The inputs are the computational pairs $(A_0,A_1)$ and $(B_0,B_1)$, representing $a = a_1/a_0$ and $b = b_1/b_0$ respectively, and the output is the computational pair $(X_0,X_1)$, representing $x = x_1/x_0 = a + b$.

\section{Feedforward composability and instantaneous correctness}
\label{sec:composability}

The main result of this paper is a composability theorem showing that ratio-encoded
reaction network modules can be assembled into feedforward computations whose
encoded values are correct for all positive times.

That result is established in this section. We begin with an abstract
single-module criterion, Lemma~\ref{lem:single_module}, which identifies the
mechanism underlying instantaneous computation at the level of ratios: an output
computational pair has a time-independent ratio whenever its production terms
have a fixed ratio and all downstream consumption enters through a common
depletion term. Next, we verify that the basic modules introduced in
Section~\ref{sec:module_intro}, namely identification, inversion,
multiplication, and addition, satisfy these conditions. Finally, in Theorem~\ref{thm:feedforward}, we show that
these same conditions imply instantaneous correctness of
arbitrary feedforward compositions of such modules.

Later, in Section~\ref{sec:alt_modules}, we present alternative constructions that
satisfy the same structural conditions, offering additional implementation
flexibility.  By Theorem~\ref{thm:feedforward}, these alternative modules may be substituted into a feedforward computation without changing the composability argument.

\subsection{Single-module ratio invariance}

We record an invariance property that is the core analytical mechanism
underlying the later compositionality results.  
In the lemma, the production
terms $F_0(t)$ and $F_1(t)$ encode the action of the computational module itself:
different choices of these terms, as functions of the input concentrations, will
realize identification, inversion, multiplication, or addition. The key point is that if the production ratio $F_1(t)/F_0(t)$ is held fixed for all positive time, and all downstream consumption enters symmetrically through the same depletion term in both equations, then the output ratio $s_1(t)/s_0(t)$ is held fixed at that same value.

\begin{lemma}[Single-module ratio invariance]
\label{lem:single_module}
Suppose that the concentrations $(s_0(t),s_1(t))^T$ of a computational pair
$(S_0,S_1)$ satisfy an ODE system of the form
\begin{equation}
\frac{d}{dt}
\begin{pmatrix}
    s_0(t) \\
    s_1(t)
\end{pmatrix}
=
\begin{pmatrix}
    F_0(t) \\
    F_1(t)
\end{pmatrix}
-
H(t)
\begin{pmatrix}
    s_0(t) \\
    s_1(t)
\end{pmatrix},
\label{eq:abstract_S}
\end{equation}
with initial conditions $s_0(0)=s_1(0)=0$.  Assume that $H,F_0,$ and $F_1$
are continuous functions of time, with $F_0(t),F_1(t)\ge 0$, and that the same
function $H(t)$ appears in both equations.

Suppose further that there exists a constant $\alpha\in[0,\infty]$ such that,
for all $t>0$,
\begin{equation}
\frac{F_1(t)}{F_0(t)}=\alpha,
\label{eq:F_ratio}
\end{equation}
where the ratio is interpreted using the extended-ratio convention.  Then, for all $t>0$,
\[
\frac{s_1(t)}{s_0(t)}=\alpha,
\]
again using the extended-ratio convention.
\end{lemma}

\begin{proof}
 First note that because the extended ratio in \eqref{eq:F_ratio} is well-defined for every $t>0$, the production pair $(F_0(t),F_1(t))$ is never $(0,0)$. Variation of constants gives
\[
s_i(t)=\int_0^t F_i(u)\exp\!\left(-\int_u^t H(r)\,dr\right)du,\qquad i=0,1.
\]
Hence $s_0(t)>0$ for every $t>0$ when $\alpha<+\infty$, while $s_1(t)>0$ for every $t>0$ when $\alpha=+\infty$.

First suppose $\alpha\in(0,\infty)$.  Define
\[
g(t):=s_1(t)-\alpha s_0(t).
\]
Using \eqref{eq:abstract_S} and \eqref{eq:F_ratio}, we obtain
\[
\dot g(t)=F_1(t)-\alpha F_0(t)-H(t)g(t)=-H(t)g(t).
\]
Since $g(0)=0$, uniqueness for this linear equation gives $g(t)\equiv 0$.
Therefore $s_1(t)=\alpha s_0(t)$ for all $t\ge 0$, and hence
\[
\frac{s_1(t)}{s_0(t)}=\alpha
\]
for all $t>0$.

If $\alpha=0$, then the extended-ratio condition in \eqref{eq:F_ratio} gives
$F_1(t)=0$ for all $t>0$.  Since $s_1(0)=0$, the second equation in
\eqref{eq:abstract_S} gives $s_1(t)\equiv 0$.  Thus the output ratio is $0$ for every $t>0$.

Similarly, if $\alpha=+\infty$, then the extended-ratio condition gives
$F_0(t)=0$ for all $t>0$.  Since $s_0(0)=0$, the first equation in
\eqref{eq:abstract_S} gives $s_0(t)\equiv 0$.  Thus the output ratio is
$+\infty$ for every $t>0$.
\end{proof}

The preceding lemma shows that a computational pair has the desired invariant
ratio whenever its production terms have a fixed ratio and all downstream
consumption enters symmetrically through a common depletion term. 
Thus, to apply
the lemma to the elementary modules in a feedforward network, we must verify two
structural properties:

\begin{enumerate}
\item The \textit{symmetric downstream consumption condition}: whenever a
computational pair $(S_0,S_1)$ is used as an input to any downstream module, the
two species are consumed in exactly the same proportion, so that all downstream
usage contributes through a single shared depletion term $H(t)$.

\item The \textit{ratio realization condition}:   the production terms for the output computational pair have a well-defined, constant extended ratio:
\[
\frac{F_1(t)}{F_0(t)}=\alpha,\qquad t>0,
\]
for some constant $\alpha \in [0,\infty]$.
\end{enumerate}

In Sections \ref{sec:cond1} and \ref{sec:cond2}, we now verify these two structural conditions for the four elementary modules
introduced in Section~\ref{sec:module_intro}.  
In doing so, we identify the
corresponding value of $\alpha$ in each case: $\alpha=a$ for identification,
$\alpha=1/a$ for inversion, $\alpha=ab$ for multiplication, and $\alpha=a+b$
for addition.

\subsubsection{Condition 1: Symmetric downstream consumption}
\label{sec:cond1}

We now verify the symmetric downstream consumption condition -- that a computational pair $(A_0,A_1)$ is symmetrically consumed when it is an input into a computational module. 
In other words, we will show that each module $m$ for which $(A_0,A_1)$ is an input contributes a term of the type 
\begin{equation}
    -h_m(t) \begin{pmatrix}
        a_0(t) \\ a_1(t)
    \end{pmatrix}. 
\end{equation}
The consumption/negative part in the complete differential equation for the pair $(A_0,A_1)$ is then obtained by
summing over all downstream modules that use $(A_0,A_1)$ as an input.  Let
\[
\mathcal M(A)
=
\{m:\text{ the computational pair }(A_0,A_1)\text{ is an input to }m\}.
\]
Then
\begin{align}
\frac{d}{dt}
\begin{pmatrix}
    a_0(t) \\ a_1(t)
\end{pmatrix}
&=
\begin{pmatrix}
    F_0(t) \\ F_1(t)
\end{pmatrix}
-
H_A(t)
\begin{pmatrix}
    a_0(t) \\ a_1(t)
\end{pmatrix},
\label{eq:complete_a}
\end{align}
where
\[
H_A(t)=\sum_{m\in\mathcal M(A)} h_m(t).
\]
Here $F_0(t)$ and $F_1(t)$ denote the production terms coming from the unique
module that produces $(A_0,A_1)$, if such a module exists; for a source pair,
these production terms are absent.

\paragraph{Identification and inversion ($a \mapsto x = a$ or $x = 1/a$).}
The relevant reactions are $A_0 \to \cdot$ and $A_1 \to \cdot$. This module contributes the following term to the ODE for $(a_0, a_1)^T$:
\begin{equation}
    -h_m(t) \begin{pmatrix}
        a_0(t) \\ a_1(t)
    \end{pmatrix} = -1 \begin{pmatrix}
        a_0(t) \\ a_1(t)
    \end{pmatrix},
\end{equation}
so that $h_m(t) = 1$.

\paragraph{Multiplication and addition ($(a,b) \mapsto x = ab$ or $x = a+b$).}
Assume that $(B_0, B_1)$ is the other computational pair that enters as a co-input with $(A_0, A_1)$. The relevant reactions are $A_0 + B_0 \to \cdot$, $A_1 + B_1 \to \cdot$, $A_0 + B_1 \to \cdot$, and $A_1 + B_0 \to \cdot$. This module contributes the following term to the ODE for $(a_0, a_1)^T$:
\begin{equation}
    -h_m(t) \begin{pmatrix}
        a_0(t) \\ a_1(t)
    \end{pmatrix} = -(b_0(t) + b_1(t)) \begin{pmatrix}
        a_0(t) \\ a_1(t)
    \end{pmatrix},
\end{equation}
so that $h_m(t) = b_0(t) + b_1(t)$.

In all cases, the consumption of $(A_0, A_1)$ by any module takes the form $-h_m(t)(a_0(t), a_1(t))^T$, confirming that Condition~1 holds for every module type.

\subsubsection{Condition 2: Ratio realization}
\label{sec:cond2}

 We now verify the ratio realization condition for each module. 
In this subsection, the pair $(X_0,X_1)$ denotes the output computational pair
of the module under consideration.  The term $H(t)$  is the symmetric downstream depletion term justified by Condition~1.  For the present verification, the
only point is to identify the production terms $F_0(t)$ and $F_1(t)$ and show
that their ratio is the intended arithmetic expression in the input ratios.

Throughout this subsection, the input values are understood to be admissible for the module under consideration.

Throughout this subsection, lowercase letters with subscripts denote
concentrations, while the corresponding letter without a subscript denotes the
encoded ratio. Thus, for an input pair $(A_0,A_1)$ we write
\[
a=\frac{a_1(t)}{a_0(t)},
\]
using the extended-ratio convention from \eqref{fractional_rep}. Similarly, for
a second input pair $(B_0,B_1)$, we write
\[
b=\frac{b_1(t)}{b_0(t)}.
\]

\paragraph{Identification ($a \mapsto x = a$).}
The module \eqref{net:identification} produces output dynamics 
\begin{equation} \label{net:identification_again}
\begin{aligned}
    \frac{d}{dt} \begin{pmatrix}
        x_0(t) \\ x_1(t) 
    \end{pmatrix}  = 
    \begin{pmatrix}
        a_0(t) \\ a_1(t) 
    \end{pmatrix} - H(t)
    \begin{pmatrix}
        x_0(t) \\ x_1(t)
    \end{pmatrix}, 
\end{aligned}
\end{equation}
where $H(t)$ is the symmetric downstream depletion term established in Condition~1. 
Assuming that $a = a_1(t)/a_0(t)$ for $t>0$, we have that $x_1(t)/x_0(t) = a$ for $t>0$ by Lemma \ref{lem:single_module}.

\paragraph{Inversion ($a \mapsto x = 1/a$).}
The module \eqref{net:inversion} produces output dynamics
\begin{equation} \label{net:inversion_again}
\begin{aligned}
    \frac{d}{dt} \begin{pmatrix}
        x_0(t) \\ x_1(t) 
    \end{pmatrix}  = 
    \begin{pmatrix}
        a_1(t) \\ a_0(t) 
    \end{pmatrix} - H(t)
    \begin{pmatrix}
        x_0(t) \\ x_1(t)
    \end{pmatrix}, 
\end{aligned}
\end{equation}
where $H(t)$  is the symmetric downstream depletion term established in Condition~1.
Assuming that $a = a_1(t)/a_0(t)$ for $t>0$, we have $F_0(t) = a_1(t)$ and $F_1(t) = a_0(t)$, so $F_1(t)/F_0(t) = a_0(t)/a_1(t) = 1/a$. 
Therefore $x_1(t)/x_0(t) = 1/a$ for $t>0$ by Lemma \ref{lem:single_module}.

\paragraph{Multiplication ($a, b \mapsto x = ab$).}
The module \eqref{net:multiplication} produces output dynamics
\begin{equation} \label{net:multiplication_again}
\begin{aligned}
    \frac{d}{dt} \begin{pmatrix}
        x_0(t) \\ x_1(t) 
    \end{pmatrix}  = 
    \begin{pmatrix}
        a_0(t) b_0(t) \\ a_1(t) b_1(t)
    \end{pmatrix} - H(t)
    \begin{pmatrix}
        x_0(t) \\ x_1(t)
    \end{pmatrix}
    =
    \begin{pmatrix} a_0(t) \\ a_1(t) \end{pmatrix} \otimes \begin{pmatrix} b_0(t) \\ b_1(t) \end{pmatrix} - H(t)
    \begin{pmatrix}
        x_0(t) \\ x_1(t)
    \end{pmatrix},
\end{aligned}
\end{equation}
where the map $\otimes$ is as defined in \eqref{def:otimes}, and $H(t)$  is the symmetric downstream depletion term established in Condition~1.
Assuming that $a = a_1(t)/a_0(t)$ and $b = b_1(t)/b_0(t)$ for $t>0$, we have $F_0(t) = a_0(t)b_0(t)$ and $F_1(t) = a_1(t)b_1(t)$, so 
\[
\frac{F_1(t)}{F_0(t)} = \frac{a_1(t) b_1(t)}{a_0(t) b_0(t)} = \frac{a_1(t)}{a_0(t)} \cdot \frac{b_1(t)}{b_0(t)} = ab.
\]
Therefore $x_1(t)/x_0(t) = ab$ for $t>0$ by Lemma \ref{lem:single_module}.

\paragraph{Addition ($a, b \mapsto x = a+b$).}
The module \eqref{net:addition} produces output dynamics
\begin{equation} \label{net:addition_again}
\begin{aligned}
    \frac{d}{dt} \begin{pmatrix}
        x_0(t) \\ x_1(t) 
    \end{pmatrix}  = 
    \begin{pmatrix}
        a_0(t) b_0(t) \\ a_0(t) b_1(t) + a_1(t) b_0(t)
    \end{pmatrix} - H(t)
    \begin{pmatrix}
        x_0(t) \\ x_1(t)
    \end{pmatrix}
    =
    \begin{pmatrix} a_0(t) \\ a_1(t) \end{pmatrix} \oplus \begin{pmatrix} b_0(t) \\ b_1(t) \end{pmatrix} - H(t)
    \begin{pmatrix}
        x_0(t) \\ x_1(t)
    \end{pmatrix},
\end{aligned}
\end{equation}
where the map $\oplus$ is as defined in \eqref{def:oplus}, and $H(t)$   is the symmetric downstream depletion term established in Condition~1.
Assuming that $a = a_1(t)/a_0(t)$ and $b = b_1(t)/b_0(t)$ for $t>0$, we have $F_0(t) = a_0(t)b_0(t)$ and $F_1(t) = a_0(t)b_1(t) + a_1(t)b_0(t)$, so
\[
\frac{F_1(t)}{F_0(t)} = \frac{a_0(t) b_1(t) + a_1(t) b_0(t)}{a_0(t) b_0(t)} = \frac{b_1(t)}{b_0(t)} + \frac{a_1(t)}{a_0(t)} = b + a.
\]
Therefore $x_1(t)/x_0(t) = a + b$ for $t>0$ by Lemma \ref{lem:single_module}.

\begin{remark}\label{rem:add_infty}
The case $a = b = +\infty$ is a notable exception: although $\infty + \infty = +\infty$ in the extended reals, the encoding gives $a_0 = b_0 = 0$, so $F_0 = F_1 = 0$ and the output is undefined. Accordingly, this case is excluded by admissibility.
\end{remark}

\subsection{Feedforward compositionality theorem}

The preceding conditions are designed precisely so that correctness propagates
through a feedforward computational graph.  The following theorem is the main theoretical result of the paper: once the source values are encoded, every downstream computational pair represents its recursively defined value at every positive time for which the network solution is defined. The theorem is independent of the particular arithmetic modules used; only Conditions~1 and~2 matter.

\begin{theorem}[Feedforward compositionality]
\label{thm:feedforward}
Let $\GG=(V,E)$ be a computational graph as defined in Definition~\ref{def:computational_graph}. Suppose:
\begin{itemize}
\item Each source vertex $v$ is associated with a computational pair $(S_{v,0}, S_{v,1})$ whose initial concentrations $s_{v,0}(0)$ and $s_{v,1}(0)$ are not both zero.  The values encoded by the source pairs form an admissible input assignment for $\GG$.

\item 
 Each non-source vertex is associated with exactly one computational module, whose input pairs are the computational pairs of its immediate predecessors in $\GG$, and every module in the graph satisfies Conditions~1 and~2.
 
\item The computational pair at each non-source vertex $v \in V$ satisfies $s_{v,0}(0) = s_{v,1}(0) = 0$.
\end{itemize}
 Then for every vertex $v\in V$, the ratio $s_{v,1}(t)/s_{v,0}(t)$ is well-defined and constant for every $t>0$ for which the network solution is defined, and equals the value obtained by recursively applying the operations represented by the modules to the source vertex ratios.
\end{theorem}

\begin{proof}
Since $\GG$ is acyclic, its vertices admit a topological ordering $v_1,\dots,v_N$ in which all source vertices appear first. The claim holds for source vertices: a source pair has no production terms, so by
Condition~1 and \eqref{eq:complete_a} its concentrations satisfy $\frac{d}{dt}(s_{v,0}, s_{v,1})^T
= -H_{v}(t)\,(s_{v,0}, s_{v,1})^T$, and so $s_{v,i}(t) = s_{v,i}(0)\,e^{-\int_0^t H_v(u)\,du}$ for
$i = 0,1$. The common positive factor $e^{-\int_0^t H_v}$ cancels in the ratio, so, since
$s_{v,0}(0)$ and $s_{v,1}(0)$ are not both zero, the ratio $s_{v,1}(t)/s_{v,0}(t)$ is well-defined
and equal to $s_{v,1}(0)/s_{v,0}(0)$ for all $t \ge 0$ for which the solution is defined.

Suppose the claim holds for $v_1,\dots,v_{k-1}$, and consider a non-source vertex $v_k$. By the inductive hypothesis, the input ratios to the module at $v_k$ are well-defined, constant for $t>0$, and equal to the values recursively determined by the source ratios at the immediate predecessors of $v_k$. 
Denote the value of the operation represented by this module, applied to these input ratios, by $\alpha_{v_k}$. By admissibility, this value is defined.

 By Condition~2 (Section~\ref{sec:cond2}), the production terms for the output pair at $v_k$ have the well-defined extended ratio $F_{v_k,1}(t)/F_{v_k,0}(t)=\alpha_{v_k}$ for every $t>0$. By Condition~1 (Section~\ref{sec:cond1}), every downstream use of the output pair contributes a term of the form $-h_m(t)(s_{v_k,0}, s_{v_k,1})^T$ to its ODE, so the complete ODE for the output pair has the form \eqref{eq:complete_a}, i.e.\
\[
\frac{d}{dt}\begin{pmatrix} s_{v_k,0}(t) \\ s_{v_k,1}(t) \end{pmatrix}
= \begin{pmatrix} F_{v_k,0}(t) \\ F_{v_k,1}(t) \end{pmatrix}
- H_{v_k}(t) \begin{pmatrix} s_{v_k,0}(t) \\ s_{v_k,1}(t) \end{pmatrix},
\]
where $H_{v_k}(t) = \sum_m h_m(t)$. Since $s_{v_k,0}(0) = s_{v_k,1}(0) = 0$, Lemma~\ref{lem:single_module} gives $s_{v_k,1}(t)/s_{v_k,0}(t) = \alpha_{v_k}$ for all $t>0$ for which the solution is defined. The result follows by induction.
\end{proof}

\begin{corollary}[Composability of the standard arithmetic modules]
\label{cor:standard_modules}
Every admissible feedforward computational graph built from the identification, inversion, multiplication, and addition modules of Section~\ref{sec:module_intro},  with every non-source computational pair
initialized at $(0,0)^T$, computes its recursively defined arithmetic values instantaneously: at every vertex, the encoded ratio is correct for all $t>0$.
\end{corollary}

\begin{proof}
Sections~\ref{sec:cond1} and~\ref{sec:cond2} verify Conditions~1 and~2 for these four modules. Proposition~\ref{prop:global} shows that the resulting mass-action solution is defined for all $t\ge0$. The conclusion therefore follows from Theorem~\ref{thm:feedforward}.
\end{proof}

Theorem~\ref{thm:feedforward} isolates the general composability principle: instantaneous ratio correctness is preserved by any feedforward collection of modules satisfying Conditions~1 and~2. Corollary~\ref{cor:standard_modules} applies this principle to the four standard arithmetic modules, for which global existence is guaranteed by Proposition~\ref{prop:global} of Appendix~\ref{app:global}.

\section{Signed dual rail representation for all real numbers}
\label{sec:real}

The ratio encoding of Sections~\ref{sec:fundamental} and \ref{sec:composability} represents elements of $[0,+\infty]$. To represent all real numbers, including negative ones, we use a \emph{signed dual rail} encoding: a quadruple of species $(A_{p0}, A_{p1}, A_{n0}, A_{n1})$, equivalently two computational pairs, representing a real number via
\[
a \coloneqq \frac{a_{p1}}{a_{p0}} - \frac{a_{n1}}{a_{n0}} \in [-\infty, \infty],  
\]
where, for example, $a_{p1}$ denotes the concentration of species $A_{p1}$. This approach is used in \cite{chen2014deterministic,chen2023rate,anderson2025arithmetic}. The expressions $0/0$ and $\infty - \infty$ are not defined, and quadruples giving rise to them are excluded: specifically, a quadruple with both $a_{p0}$ and $a_{p1}$ equal to zero, or both $a_{n0}$ and $a_{n1}$ equal to zero, or both $a_{p0}$ and $a_{n0}$ equal to zero is left undefined. All other quadruples have unambiguous meaning; for instance, if $a_{n0} = 0$ with $a_{n1} \ne 0$ and $a_{p0} \ne 0$, then $a = -\infty$. 

Since this representation of a real number is not unique, we introduce a
canonical signed dual rail encoding.

\begin{definition}[Canonical signed real]
A signed dual rail encoding $a = (a_p, a_n)$ with $a_p, a_n \in [0,+\infty]$ is \emph{canonical} if $a_p = 0$ or $a_n = 0$ (or both). In canonical form, every element of $[-\infty,+\infty]$ has a  unique pair of rail values: positive $a$ is encoded as $(a, 0)$, negative $a$ as $(0,-a)$, zero as $(0,0)$, $+\infty$ as $(+\infty, 0)$, and $-\infty$ as $(0, +\infty)$.
\end{definition}

 The uniqueness here is at the level of the rail values $(a_p,a_n)$, not at the level of the four concentrations: each individual rail value retains the nonuniqueness of the underlying ratio encoding.

It is natural to ask whether starting from canonical inputs the outputs of the real modules are also canonical. As we show below, multiplication preserves canonical form but addition does not in general. 

\subsection{Real addition module}

In order to define addition of two extended real numbers $a,b \in [-\infty, \infty]$, we note that 
\begin{align*}
    a + b &= \left(\frac{a_{p1}}{a_{p0}} - \frac{a_{n1}}{a_{n0}}\right) + \left(\frac{b_{p1}}{b_{p0}} - \frac{b_{n1}}{b_{n0}}\right) \\
    &= \left(\frac{a_{p1}}{a_{p0}} + \frac{b_{p1}}{b_{p0}} \right) - \left(\frac{a_{n1}}{a_{n0}} + \frac{b_{n1}}{b_{n0}}\right).
\end{align*}
The positive and negative output rails are therefore
\[
\begin{pmatrix} a_{p0} \\ a_{p1} \end{pmatrix} \oplus \begin{pmatrix} b_{p0} \\ b_{p1} \end{pmatrix}
\quad \text{and} \quad
\begin{pmatrix} a_{n0} \\ a_{n1} \end{pmatrix} \oplus \begin{pmatrix} b_{n0} \\ b_{n1} \end{pmatrix},
\]
respectively. In other words, real addition is simply ordinary addition on both positive and negative rails independently.

\begin{remark}\label{rem:add_canonical}
Real addition inherits the $\infty + \infty$ limitation of the non-negative module (Remark~\ref{rem:add_infty}): the cases $a = b = +\infty$ and $a = b = -\infty$ leave the respective output rails undefined,  while $a=+\infty$ and $b=-\infty$ correspond to the undefined extended-real sum $+\infty+(-\infty)$: both output rails encode $+\infty$, producing the undefined signed difference $+\infty-\infty$. Additionally, canonical form is not preserved in general: when $a$ and $b$ have opposite signs, both output rails may be nonzero (e.g.\ $a = 3$, $b = -2$ gives output rails $(3, 2)$), and restoring canonical form requires a subtraction step outside the computational design of this paper.
\end{remark}

\subsection{Real multiplication module}

To define multiplication of two extended real numbers $a,b \in [-\infty, \infty]$, note that 
   \begin{align*}
    a \cdot b &= \left(\frac{a_{p1}}{a_{p0}} - \frac{a_{n1}}{a_{n0}}\right) \cdot \left(\frac{b_{p1}}{b_{p0}} - \frac{b_{n1}}{b_{n0}}\right) \\
    &= \left(\frac{a_{p1}b_{p1}}{a_{p0}b_{p0}} + \frac{a_{n1}b_{n1}}{a_{n0}b_{n0}} \right) - \left(\frac{a_{p1}b_{n1}}{a_{p0}b_{n0}} + \frac{a_{n1}b_{p1}}{a_{n0}b_{p0}} \right).
\end{align*}
The positive and negative output rails are therefore
\[
\left(\begin{pmatrix} a_{p0} \\ a_{p1} \end{pmatrix} \otimes \begin{pmatrix} b_{p0} \\ b_{p1} \end{pmatrix}\right) \oplus \left(\begin{pmatrix} a_{n0} \\ a_{n1} \end{pmatrix} \otimes \begin{pmatrix} b_{n0} \\ b_{n1} \end{pmatrix}\right)
\quad \text{and} \quad
\left(\begin{pmatrix} a_{p0} \\ a_{p1} \end{pmatrix} \otimes \begin{pmatrix} b_{n0} \\ b_{n1} \end{pmatrix}\right) \oplus \left(\begin{pmatrix} a_{n0} \\ a_{n1} \end{pmatrix} \otimes \begin{pmatrix} b_{p0} \\ b_{p1} \end{pmatrix}\right),
\]
respectively. A real multiplication is thus composed of four non-negative multiplications and two non-negative additions.

\begin{remark}\label{rem:mult_canonical}
For well-defined signed dual rail inputs, real multiplication introduces no new undefined cases beyond those of the non-negative modules. Moreover, if both inputs are in canonical form, the output is also canonical: since at most one of $\{a_p, a_n\}$ and one of $\{b_p, b_n\}$ is nonzero, the positive and negative output rails cannot both be nonzero simultaneously.
\end{remark}

\subsection{Real identification module}

Real identification is straightforward and does not require any new ideas. Since the positive and negative rails are independent computational pairs, it suffices to apply the non-negative identification module \eqref{net:identification} to each rail in parallel:
\begin{equation}
\begin{aligned}
    A_{p0} \to X_{p0}, \quad A_{p1} \to X_{p1}, \quad
    A_{n0} \to X_{n0}, \quad A_{n1} \to X_{n1}.
\end{aligned}
\end{equation}

\subsection{Real inversion module}

Real inversion --- computing $1/a$ for $a \in [-\infty,\infty]$ --- does not decompose cleanly over the signed dual rail encoding. Writing $a = a_p - a_n$ where $a_p = a_{p1}/a_{p0}$ and $a_n = a_{n1}/a_{n0}$, the reciprocal $1/(a_p - a_n)$ has no simple expression in terms of $1/a_p$ and $1/a_n$ alone. We therefore leave a general real inversion module as an open problem for future work. We note that real inversion is not required for the power series and matrix multiplication applications developed in Section~\ref{sec:power_series} onwards, which use only addition and multiplication.

The notion of admissibility extends naturally to signed computations:
a signed computational graph is admissible if each signed module is defined
on the values it receives.

\begin{corollary}[Signed feedforward composability]
\label{cor:signed_feedforward}
Let $\GG$ be a finite admissible feedforward computational graph whose source
values lie in $[-\infty,\infty]$ and whose non-source vertices use real
identification, addition, or multiplication as defined above. Then $\GG$ can
be realized by an admissible feedforward graph of the nonnegative modules of
Section~\ref{sec:fundamental}, with every non-source computational pair
initialized at $(0,0)^T$. Consequently, every signed dual rail output
represents its recursively defined value for all $t>0$.
\end{corollary}

\begin{proof}
Replace each real identification, addition, and multiplication vertex by the
corresponding collection of nonnegative modules described above, and initialize
every non-source computational pair in the resulting graph at $(0,0)^T$. By
admissibility, every constituent rail operation is defined, so the resulting
nonnegative feedforward graph is admissible. Corollary~\ref{cor:standard_modules}
therefore gives the correct ratio on every positive and negative rail for all
$t>0$. Taking the difference of the two rail values gives the claimed signed
value.
\end{proof}

\section{Applications: power series and matrix computations}
\label{sec:power_series}

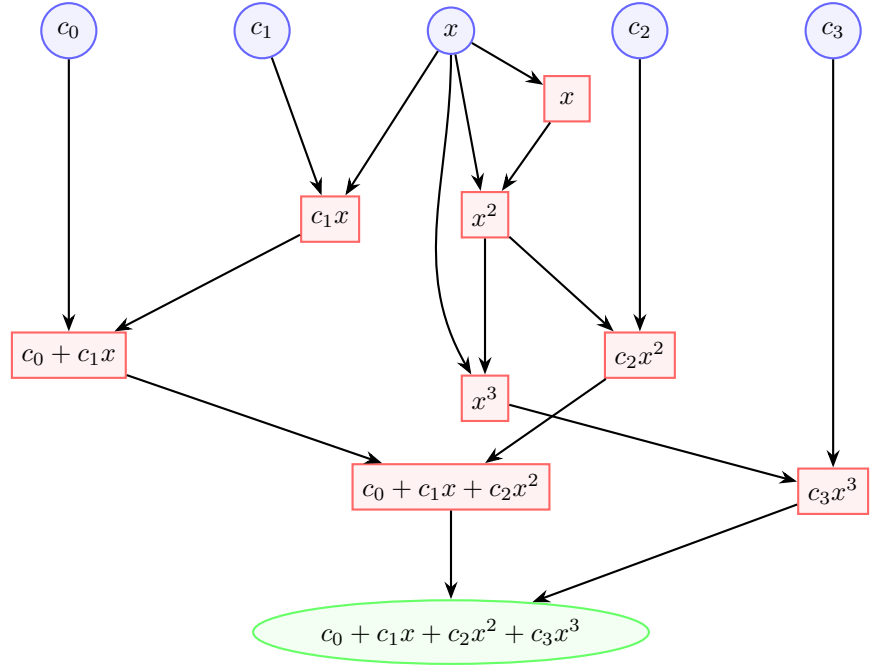
\begin{figure}[h!]
\centering
\begin{tikzpicture}[
    varnode/.style={circle, draw=blue!60, fill=blue!5, thick, minimum size=6mm, font=\small},
    opnode/.style={rectangle, draw=red!60, fill=red!5, thick, minimum size=6mm, font=\small},
    finalnode/.style={ellipse, draw=green!60, fill=green!5, thick, minimum size=6mm, font=\small},
    arrow/.style={-Stealth, thick}
]

\node[varnode] (a0) {$c_0$};
\node[varnode, right=1.8cm of a0] (a1) {$c_1$};
\node[varnode, right=1.8cm of a1] (x) {$x$};
\node[varnode, right=1.8cm of x] (a2) {$c_2$};
\node[varnode, right=1.8cm of a2] (a3) {$c_3$};

\node[opnode, right=0.9cm of x, yshift=-0.9cm] (x-copy) {$x$};

\node[opnode, below=1.8cm of x, xshift=0.45cm] (x2) {$x^2$};

\node[opnode, below=1.8cm of x2] (x3) {$x^3$};

\node[opnode, below=1.8cm of a1, xshift=0.9cm] (a1x) {$c_1 x$};
\node[opnode, below=3.6cm of a2] (a2x2) {$c_2 x^2$};
\node[opnode, below=5.4cm of a3] (a3x3) {$c_3 x^3$};

\node[opnode, below=3.6cm of a0] (sum1) {$c_0 + c_1 x$};
\node[opnode, below=5.4cm of x] (sum2) {$c_0 + c_1 x + c_2 x^2$};
\node[finalnode, below=7.2cm of x] (result) {$c_0 + c_1 x + c_2 x^2 + c_3 x^3$};

\draw[arrow] (x) -- (x-copy);

\draw[arrow] (x) -- (x2);
\draw[arrow] (x-copy) -- (x2);

\draw[arrow] (x2) -- (x3);
\draw[arrow] (x) to[out=-90, in=120] (x3);

\draw[arrow] (a1) -- (a1x);
\draw[arrow] (x) -- (a1x);

\draw[arrow] (a2) -- (a2x2);
\draw[arrow] (x2) -- (a2x2);

\draw[arrow] (a3) -- (a3x3);
\draw[arrow] (x3) -- (a3x3);

\draw[arrow] (a0) -- (sum1);
\draw[arrow] (a1x) -- (sum1);

\draw[arrow] (sum1) -- (sum2);
\draw[arrow] (a2x2) -- (sum2);

\draw[arrow] (sum2) -- (result);
\draw[arrow] (a3x3) -- (result);

\end{tikzpicture}
\caption{Computational graph for computing the truncated power series $c_0 + c_1x + c_2x^2 + c_3x^3$. The inputs are $x$ and coefficients $c_0, c_1, c_2, c_3$. The computation uses identification, multiplication, and addition. Note that identification is a convenient way to produce a copy of $x$ in order to compute $x^2$.}
\label{fig:power_series}
\end{figure}

We record three representative applications: truncated power series, matrix multiplication, and determinants. Each is a finite feedforward computation using only identification, addition, and multiplication. For real-valued inputs, instantaneous correctness therefore follows directly from Corollary~\ref{cor:signed_feedforward}; for nonnegative inputs this reduces to Corollary~\ref{cor:standard_modules}.

Truncated power series give a simple first application. Since a polynomial can
be evaluated using only identification, multiplication, and addition, its
computational graph can be implemented directly by the modules constructed
above. Figure~\ref{fig:power_series} illustrates the case of a cubic polynomial. Note that a pair may serve as an input to any number of downstream modules, since Condition~1 simply sums the resulting depletion terms. If a single module is to take the same pair as both of its inputs, as when squaring a value, an identification module provides an easy way to supply the required copy.

\begin{corollary}[Instantaneous power series computation]
 Let $x,c_0,\ldots,c_n\in\R$ be encoded in signed dual rail form. Then the truncated power series $p(x)=\sum_{k=0}^{n}c_kx^k$ can be computed instantaneously by a reaction network implementing a feedforward computational graph built from real identification, multiplication, and addition modules. By Corollary~\ref{cor:signed_feedforward}, the signed output represents $p(x)$ for all $t>0$.
\end{corollary}

Matrix multiplication provides another canonical feedforward computation. Each
entry of the product matrix is a finite sum of pairwise products, and hence can
be implemented using only multiplication and addition modules.

\begin{corollary}[Instantaneous matrix multiplication]
Given matrices $M \in \R^{m \times p}$ and $N \in \R^{p \times n}$, the product
$P=MN$ can be computed instantaneously by a reaction network implementing the
corresponding feedforward computational graph. The $(i,j)$ entry is
\[
P_{ij}=\sum_{k=1}^p M_{ik}N_{kj},
\]
which requires $p$ real multiplications and $p-1$ real additions. Thus the full matrix product can be implemented using $mnp$ real multiplication modules and $mn(p-1)$ real addition modules. By Corollary~\ref{cor:signed_feedforward}, all $mn$ output entries are computed instantaneously and simultaneously.
\end{corollary}

Determinants follow in the same way once signed values are available.

\begin{corollary}[Instantaneous determinant]
Let $M\in\R^{n\times n}$ have entries encoded in the signed dual rail form of
Section~\ref{sec:real}. Then
\[
\det M=\sum_{\sigma\in S_n}\sgn(\sigma)\prod_{i=1}^{n}M_{i,\sigma(i)}
\]
can be computed instantaneously by the reaction network implementing the corresponding
feedforward graph, using $n!\,(n-1)$ real multiplication modules and $n!-1$ real addition
modules, where a factor $\sgn(\sigma)=-1$ is implemented by exchanging the positive and
negative rails of the corresponding product.  The conclusion follows from Corollary~\ref{cor:signed_feedforward}.
\end{corollary}

\section{Alternative module designs}
\label{sec:alt_modules}

The two structural conditions verified in Sections~\ref{sec:cond1} and~\ref{sec:cond2} --- the symmetric downstream consumption condition and the ratio realization condition --- admit many solutions beyond the standard modules defined in Section~\ref{sec:fundamental}. By Theorem~\ref{thm:feedforward}, any such module can be used in a feedforward computation while preserving instantaneous ratio correctness. This flexibility means the framework is not tied to a single implementation and can accommodate physical or engineering constraints --- in particular, the requirement that certain input species not be consumed by a module. We illustrate this with alternative designs for both the addition and multiplication modules, verifying both structural conditions explicitly in each case.

\subsection{Alternative addition modules}
\label{sec:alt_add}

The following module performs addition $(a,b)\mapsto x = a+b$ without consuming the $A$ input:
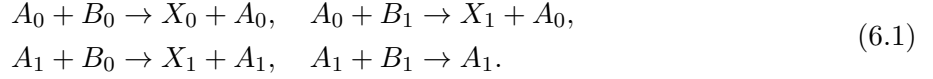
\begin{equation} \label{net:addition_alt1}
\begin{aligned}
    A_0 + B_0 &\to X_0 + A_0, \quad 
    A_0 + B_1 \to X_1 + A_0, \\
    A_1 + B_0 &\to X_1 + A_1, \quad 
    A_1 + B_1 \to A_1.  
\end{aligned}
\end{equation}

\paragraph{Condition 1.}
Since $A$ is returned as a product in every reaction of \eqref{net:addition_alt1}, it is not consumed by this module, so this module contributes $h_m(t) = 0$ to the ODE for $(a_0,a_1)^T$. For the $B$ pair, each molecule of $B_0$ or $B_1$ is consumed at rate $a_0(t)+a_1(t)$, giving $h_m(t) = a_0(t)+a_1(t)$, the same as in the standard addition module. In both cases the consumption is symmetric across the two species in each pair, confirming Condition~1.

\paragraph{Condition 2.}
The output dynamics are
\begin{equation}
\frac{d}{dt} \begin{pmatrix} x_0(t) \\ x_1(t) \end{pmatrix}
= \begin{pmatrix} a_0(t) b_0(t) \\ a_0(t) b_1(t) + a_1(t) b_0(t) \end{pmatrix}
- H(t) \begin{pmatrix} x_0(t) \\ x_1(t) \end{pmatrix}
= \begin{pmatrix} a_0(t) \\ a_1(t) \end{pmatrix} \oplus \begin{pmatrix} b_0(t) \\ b_1(t) \end{pmatrix}
- H(t) \begin{pmatrix} x_0(t) \\ x_1(t) \end{pmatrix},
\end{equation}
which is identical to the standard addition module \eqref{net:addition_again}. Therefore Condition~2 holds and $x_1(t)/x_0(t) = a+b$ for all $t>0$ by Lemma~\ref{lem:single_module}.

\hspace{.1in}

For a construction in which neither input is consumed, one could use:
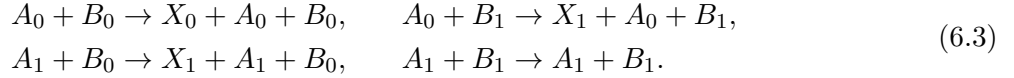
\begin{equation} \label{net:addition_alt2}
\begin{aligned}
    &A_0 + B_0 \to X_0 + A_0 + B_0, \quad 
    &&A_0 + B_1 \to X_1 + A_0 + B_1, \\
    &A_1 + B_0 \to X_1 + A_1 + B_0, \quad
    &&A_1 + B_1 \to A_1 + B_1.  
\end{aligned}
\end{equation}
The output dynamics are again identical to \eqref{net:addition_again}, so Condition~2 holds. For Condition~1, both $A$ and $B$ are returned as products, so this module contributes $h_m(t) = 0$ to the ODEs for both $(a_0,a_1)^T$ and $(b_0,b_1)^T$.

\subsection{Alternative multiplication modules}
\label{sec:alt_mult}

An alternative multiplication module in which input $A$ is not consumed is:
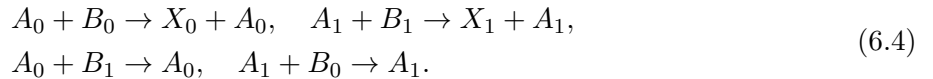
\begin{equation} \label{net:multiplication_alt1}
\begin{aligned}
    A_0 + B_0 &\to X_0 + A_0, \quad
    A_1 + B_1 \to X_1 + A_1, \\
    A_0 + B_1 &\to A_0, \quad
    A_1 + B_0 \to A_1.
\end{aligned}
\end{equation}

\paragraph{Condition 1.}
Since $A$ is returned as a product in every reaction of \eqref{net:multiplication_alt1}, it is not consumed, so this module contributes $h_m(t) = 0$ to the ODE for $(a_0,a_1)^T$. For the $B$ pair, each molecule of $B_0$ or $B_1$ is consumed at rate $a_0(t)+a_1(t)$, giving $h_m(t) = a_0(t)+a_1(t)$, the same as in the standard multiplication module. Condition~1 holds in both cases.

\paragraph{Condition 2.}
The output dynamics are
\begin{equation}
\frac{d}{dt} \begin{pmatrix} x_0(t) \\ x_1(t) \end{pmatrix}
= \begin{pmatrix} a_0(t) b_0(t) \\ a_1(t) b_1(t) \end{pmatrix}
- H(t) \begin{pmatrix} x_0(t) \\ x_1(t) \end{pmatrix}
= \begin{pmatrix} a_0(t) \\ a_1(t) \end{pmatrix} \otimes \begin{pmatrix} b_0(t) \\ b_1(t) \end{pmatrix}
- H(t) \begin{pmatrix} x_0(t) \\ x_1(t) \end{pmatrix},
\end{equation}
which is identical to the standard multiplication module \eqref{net:multiplication_again}. Therefore Condition~2 holds and $x_1(t)/x_0(t) = ab$ for all $t>0$ by Lemma~\ref{lem:single_module}.

\vspace{.1in}

For a construction in which neither input is consumed, one could use:
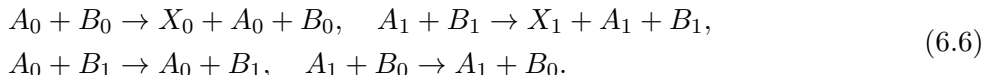
\begin{equation} \label{net:multiplication_alt2}
\begin{aligned}
    A_0 + B_0 &\to X_0 + A_0 + B_0, \quad
    A_1 + B_1 \to X_1 + A_1 + B_1, \\
    A_0 + B_1 &\to A_0 + B_1, \quad
    A_1 + B_0 \to A_1 + B_0.
\end{aligned}
\end{equation}
The output dynamics are again identical to \eqref{net:multiplication_again}, so Condition~2 holds. Both $A$ and $B$ are returned as products, so this module contributes $h_m(t) = 0$ to the ODEs for both $(a_0,a_1)^T$ and $(b_0,b_1)^T$, confirming Condition~1.
We note that the non-consuming variants introduce no finite-time blow-up: as shown in Proposition~\ref{prop:global} of Appendix~\ref{app:global}, every network built from the modules of this paper --- standard or alternative, with or without decay --- has a solution defined for all $t \ge 0$.

\subsection{Unbounded growth and decay reactions}
\label{sec:decay}

In all non-consuming constructions, the output species $(X_0, X_1)$ are produced but never consumed by the module itself, so their concentrations can grow without bound if $H(t) = 0$ (i.e., if the output pair is not consumed by any downstream module). To prevent this while preserving both structural conditions, one may add decay reactions
\[
X_0 \to 0, \qquad X_1 \to 0.
\]
These reactions consume $X_0$ and $X_1$ symmetrically at rate $1$ per molecule, contributing $h_m(t) = 1$ to $H(t)$ in both equations for $(x_0,x_1)^T$. These decay reactions contribute the same type of symmetric depletion term as an identification or inversion module, and Lemma~\ref{lem:single_module} still applies.
 The ratio $x_1(t)/x_0(t)$ is therefore unaffected.

\section{Discussion}
\label{sec:discussion}

The central contribution of this paper is a reaction network framework that computes arithmetic operations \emph{instantaneously}, meaning that the output ratio $x_1(t)/x_0(t)$ equals the correct value for all $t > 0$, not merely in the limit as $t \to \infty$. This stands in contrast to our earlier work \cite{anderson2025arithmetic}, where computations are completed within a finite time that is independent of the inputs but still positive. 
Thus, at the level of the encoded ratio, the output is correct for every $t > 0$.

However, an honest accounting of the overall computation time must include the cost of decoding the output. The reaction network encodes all values as ratios of computational pairs, and the final answer is stored as the ratio $x_1(t)/x_0(t)$ of two chemical concentrations --- or, in the signed case, as a quadruple representing $x_{p1}/x_{p0} - x_{n1}/x_{n0}$. Extracting a usable number requires a separate decoding step whose complexity depends on the encoding used. For non-negative computations over $[0,+\infty]$, decoding requires computing a single ratio $x_1/x_0$, together with the convention that $x_0 = 0$ is interpreted as $+\infty$. For signed computations over $[-\infty,+\infty]$, decoding requires two divisions, a subtraction, and appropriate handling of the cases where one or both rails encode $\pm\infty$. The total speed of the computation depends critically on how this decoding step is performed.

Two natural scenarios arise. In the first, the decoding is itself performed chemically. The operations needed --- division and rectified subtraction (i.e.\ $\max(a-b,0)$ and $\max(b-a,0)$, which together reconstruct a signed difference) --- are available at input-independent speed via \cite{anderson2025arithmetic}, making that paper a natural companion for the chemical decoding step. 
In the second scenario, the decoding is performed on a separate digital platform, which measures the relevant chemical concentrations and computes the required ratios and differences directly. 
In this case the post-measurement arithmetic is negligible; the remaining practical requirement is an efficient analog-digital interface capable of resolving the relevant concentrations. 
The winning solution, among the two scenarios, depends on the application and the hardware available.

\paragraph{Ratio correctness versus physical readout.}
The composability theorem guarantees that the encoded ratio is correct for every $t>0$, but this does not by itself guarantee that the ratio can be measured accurately at arbitrarily small times. Since every non-source computational pair is initialized at $(0,0)$, the concentrations representing the correct ratio may initially be very small.

To make this distinction explicit, suppose a finite nonnegative output value $r$ satisfies
\[
x_1(t)=r x_0(t),
\]
and that the measured concentrations are
\[
\widehat{x}_i(t)=x_i(t)+e_i(t),\qquad |e_i(t)|\le \eta,\qquad i=0,1.
\]
If the measured ratio is
\[
\widehat r=\frac{\widehat{x}_1(t)}{\widehat{x}_0(t)}
\]
and $x_0(t)>\eta$, then
\[
|\widehat r-r|
\le
\frac{(1+r)\eta}{x_0(t)-\eta}.
\]
Thus, if we want $|\widehat r-r|\le\varepsilon$ for some $\varepsilon>0$, then it is sufficient that
\[
x_0(t)\ge \eta+\frac{(1+r)\eta}{\varepsilon}.
\]

Thus, although the encoded ratio is mathematically correct for every positive time, accurate physical readout may require waiting until the relevant concentrations are sufficiently large compared with the measurement error. Quantifying this readout requirement as a function of graph structure, module design, and measurement precision is a separate problem.

Speed, however, is not the only resource that matters in chemical computation. The constructions in this paper optimize on time at the cost of other resources, and this trade-off deserves attention. Computing a function via a truncated power series requires committing to a fixed level of approximation in advance. More significantly, each additional term in the power series requires additional arithmetic operations and therefore additional chemical species --- the number of species grows with the degree of the approximation. In applications where the number of available species is limited, or where a simpler implementation is preferred, it may be more practical to employ specialized algorithms designed for specific functions. For example, the constructions in \cite{anderson2026computing} for the exponential and logarithm use a small fixed number of chemical species regardless of the desired accuracy, at the cost of giving up instantaneous computation. The right choice between these approaches depends on the specific constraints of the application.

\subsection*{Acknowledgments}

DFA gratefully acknowledges support from the Trustees of the William F. Vilas Estate, and via NSF grant DMS-2051498.  Support for this research was also provided by the University of Wisconsin-Madison, Office of the Vice Chancellor for Research with funding from the Wisconsin Alumni Research Foundation.
BJ is grateful for support from AMS-Simons Research Enhancement Grants for Primarily Undergraduate Institution (PUI) and NSF DMS-2051498.

\appendix

\section{Global existence}
\label{app:global}

\begin{proposition}[Global existence]
\label{prop:global}
Let $\GG$ be a feedforward computational graph in which each non-source vertex is implemented by one of the
modules of Section~\ref{sec:fundamental} or by one of the alternative modules of
Section~\ref{sec:alt_modules}, with or without the decay reactions of Section~\ref{sec:decay}.
Then, for every nonnegative initial condition, the associated mass-action system has a unique
solution, and this solution remains nonnegative and finite on all of $[0,\infty)$.
\end{proposition}

\begin{proof}
    The vector field is polynomial, hence locally Lipschitz, so there is a unique solution on a maximal interval $[0,t_{\max})$.  The nonnegative orthant is forward invariant because every negative term in the ODE of a species contains that species as a factor.

    Suppose, toward a contradiction, that $t_{\max}<\infty$. Order the vertices of $\GG$ topologically. We show inductively that every concentration is bounded on $[0,t_{\max})$.

For a source pair, $F_0=F_1=0$, and hence
\[
z_i'(t)=-H(t)z_i(t)\le 0,
\]
so $z_i(t)\le z_i(0)$ for $i=0,1$.

Now consider a non-source vertex and suppose that all concentrations at its
immediate predecessors are bounded on $[0,t_{\max})$. Choose $M\ge 1$ that
bounds all of these concentrations. By Condition~1, its output pair satisfies
\[
\frac{d}{dt}
\begin{pmatrix}z_0(t)\\z_1(t)\end{pmatrix}
=
\begin{pmatrix}F_0(t)\\F_1(t)\end{pmatrix}
-
H(t)\begin{pmatrix}z_0(t)\\z_1(t)\end{pmatrix},
\]
with $H(t)\ge0$. For identification and inversion, each $F_i$ is bounded by
$M$; for multiplication, by $M^2$; and for addition, by $2M^2$. Thus in every
case $F_i(t)\le2M^2$, and therefore
\[
z_i(t)\le z_i(0)+\int_0^t F_i(u)\,du
\le z_i(0)+2M^2t_{\max},
\qquad 0\le t<t_{\max}.
\]
Hence the output pair is bounded on $[0,t_{\max})$.

Proceeding through the finite topological ordering shows that every
concentration is bounded on $[0,t_{\max})$, contradicting maximality of
$t_{\max}$. Therefore $t_{\max}=\infty$.
\end{proof}

The same induction run with time-dependent bounds shows that, for each fixed computational graph, every concentration grows at most polynomially in time, with the degree depending on the depth and structure of $\GG$.

\bibliographystyle{unsrt}
\bibliography{frac}

\begin{thebibliography}{10}

\bibitem{anderson2021reaction}
David~F. Anderson, Badal Joshi, and Abhishek Deshpande.
\newblock On reaction network implementations of neural networks.
\newblock {\em Journal of the Royal Society Interface}, 18(177):20210031, 2021.

\bibitem{cappelletti2020stochastic}
Daniele Cappelletti, Andres Ortiz-Mu{\~n}oz, David~F. Anderson, and Erik
  Winfree.
\newblock Stochastic chemical reaction networks for robustly approximating
  arbitrary probability distributions.
\newblock {\em Theoretical Computer Science}, 801:64--95, 2020.

\bibitem{chen2023rate}
Ho-Lin Chen, David Doty, Wyatt Reeves, and David Soloveichik.
\newblock Rate-independent computation in continuous chemical reaction
  networks.
\newblock {\em Journal of the ACM}, 70(3):1--61, 2023.

\bibitem{chen2014deterministic}
Ho-Lin Chen, David Doty, and David Soloveichik.
\newblock Deterministic function computation with chemical reaction networks.
\newblock {\em Natural Computing}, 13:517--534, 2014.

\bibitem{qian2011neural}
Lulu Qian, Erik Winfree, and Jehoshua Bruck.
\newblock {Neural network computation with DNA strand displacement cascades}.
\newblock {\em Nature}, 475(7356):368--372, 2011.

\bibitem{QSW2011}
Lulu Qian, David Soloveichik, and Erik Winfree.
\newblock Efficient {T}uring-universal computation with {DNA} polymers.
\newblock In {\em {DNA Computing and Molecular Programming 16}}, pages
  123--140, 2011.

\bibitem{qian2011simple}
Lulu Qian and Erik Winfree.
\newblock {A simple DNA gate motif for synthesizing large-scale circuits}.
\newblock {\em J. R. Soc. Interface}, 8(62):1281--1297, 2011.

\bibitem{soloveichik2010dna}
David Soloveichik, Georg Seelig, and Erik Winfree.
\newblock {DNA} as a universal substrate for chemical kinetics.
\newblock {\em Proceedings of the National Academy of Sciences},
  107(12):5393--5398, 2010.

\bibitem{qian2011scaling}
Lulu Qian and Erik Winfree.
\newblock {Scaling up digital circuit computation with DNA strand displacement
  cascades}.
\newblock {\em Science}, 332(6034):1196--1201, 2011.

\bibitem{fages2017strong}
Fran{\c{c}}ois Fages, Guillaume Le~Guludec, Olivier Bournez, and Amaury Pouly.
\newblock Strong {T}uring completeness of continuous chemical reaction networks
  and compilation of mixed analog-digital programs.
\newblock In {\em International Conference on Computational Methods in Systems
  Biology}, pages 108--127. Springer, 2017.

\bibitem{doty2025analog}
David Doty, Mina Latifi, and David Soloveichik.
\newblock Analog computation with transcriptional networks.
\newblock {\em arXiv preprint arXiv:2508.14017}, 2025.

\bibitem{anderson2025arithmetic}
David~F. Anderson and Badal Joshi.
\newblock Chemical mass-action systems as analog computers: Implementing
  arithmetic computations at specified speed.
\newblock {\em Theoretical Computer Science}, 1025:114983, 2025.

\bibitem{anderson2026computing}
David~F. Anderson, Badal Joshi, and Tung~D. Nguyen.
\newblock Computing with reaction networks at input-independent speed:
  exponential and logarithmic functions.
\newblock {\em arXiv preprint arXiv:2604.08859}, 2026.

\bibitem{salehi2017chemical}
Sayed~Ahmad Salehi, Keshab~K Parhi, and Marc~D. Riedel.
\newblock Chemical reaction networks for computing polynomials.
\newblock {\em ACS Synthetic Biology}, 6(1):76--83, 2017.

\bibitem{salehi2018computing}
Sayed~Ahmad Salehi, Xingyi Liu, Marc~D. Riedel, and Keshab~K Parhi.
\newblock Computing mathematical functions using {DNA} via fractional coding.
\newblock {\em Scientific Reports}, 8(1):8312, 2018.

\bibitem{solanki2023computing}
Arnav Solanki, Tonglin Chen, and Marc Riedel.
\newblock Computing mathematical functions with chemical reactions via
  stochastic logic.
\newblock {\em PLOS ONE}, 18(5):e0281574, 2023.

\end{thebibliography}

\end{document}